\documentclass[12pt]{amsart}
\usepackage{amssymb,amsmath,amsthm,mathrsfs,multirow,enumerate}

\usepackage{graphicx} 
\usepackage[all]{xy}
\usepackage{cite}
\usepackage{color}
\usepackage[bottom=1.5in, top=1.5in, left=1 in, right=1 in]{geometry}

\numberwithin{equation}{section} 

\theoremstyle{plain}
\newtheorem{Theorem}{Theorem}[section]
\newtheorem{Corollary}[Theorem]{Corollary}

\theoremstyle{definition}

\newtheorem{Example}{Example}

\newtheorem*{Remark}{Remark}

\newcommand{\bbm}{\begin{bmatrix}}
\newcommand{\ebm}{\end{bmatrix}}

\begin{document}
\title[Hilbert series of typical representations for Lie superalgebras] {Hilbert series of typical representations for Lie superalgebras} 
\author{Alexander Heaton and Songpon Sriwongsa}

\address{Alexander Heaton\\ (1) Max Planck Institute for Mathematics in the Sciences\\ Leipzig, Germany \\ and (2) Technische Universit\"at Berlin}
\email{\tt heaton@mis.mpg.de, alexheaton2@gmail.com}

\address{Songpon Sriwongsa\\ (1) Department of Mathematics \\ Faculty of Science\\ King Mongkut's University of Technology Thonburi (KMUTT) \\Bangkok 10140, Thailand, 
	\newline (2) Center of excellence in theoretical and computational science (TaCS-COE), Science Laboratory Building, Faculty of Science, King Mongkut's University of Technology Thonburi (KMUTT), Bangkok 10140, Thailand} 
\email{\tt songpon.sri@kmutt.ac.th, songponsriwongsa@gmail.com}


\subjclass[2010]{}

\date{\today}

\maketitle

\begin{abstract}
Let $\mathfrak{g}$ be a basic classical Lie superalgebra over $\mathbb{C}$. In the case of a typical weight whose every nonnegative integer multiple is also typical, we compute a closed form for the Hilbert series whose coefficients encode the dimensions of finite-dimensional irreducible typical $\mathfrak{g}$-representations. We give a formula for this Hilbert series in terms of elementary symmetric polynomials and Eulerian polynomials. Additionally, we show a simple closed form in terms of differential operators.
\end{abstract}

\section{Introduction}

The purpose of this paper is to produce generating functions applicable to Lie superalgebras, analogous to those that exist for Lie algebras, along the lines of \cite{GW11}. We consider highest weights $\Lambda$ of the basic classical Lie superalgebras which are dominant, integral, and whose every multiple $n\Lambda, \, n \in \mathbb{N}$ is typical. By the Hilbert series $H_\Lambda(q)$ we mean the formal power series recording the dimensions of the irreducible finite-dimensional typical representations $V(n\Lambda)$ associated to the highest weight $\Lambda$ as in
\begin{equation*}
    H_\Lambda(q) = \sum_{n \in \mathbb{N}} \text{dim} V(n\Lambda) \,\, q^n.
\end{equation*}
The main results Theorem \ref{mainthm} and Corollary \ref{thm:main-theorem} give two formulas for this Hilbert series: one as a series in terms of finitely many elementary symmetric polynomials and Eulerian polynomials and the other as
\begin{equation*}
    H_\Lambda (q) = h_{\Lambda}\left(q \frac{d}{dq}\right) \frac{1}{1-q},
\end{equation*}
where $h_{\Lambda}$ is a polynomial in one variable evaluated symbolically at the differential operator given above. The polynomial $h_\Lambda$ factors as a product of affine-linear forms over $\mathbb{Q}$. Explicitly, we will show that
\begin{equation*}
    h_{\Lambda}(t) = 2^{|\Delta_1^+|} \prod_{\alpha \in \Delta_0^+} \left( 1 - \frac{(\rho_1,\alpha)}{(\rho_0,\alpha)} + \frac{(\Lambda,\alpha)}{(\rho_0,\alpha)} \, t \right),
\end{equation*}
where all parameters are described in Section \ref{Pre}.
Our result is purely combinatorial. \color{black} However, geometric interpretations of the formula may also be of interest. In Section \ref{section:grassmannian} we discuss a geometric example which would motivate extending the formulas we derive to the case of singly atypical representations. Section \ref{Pre} gives some basic preliminaries while Section \ref{section:main-results} proves our main results. Finally, Section \ref{section:examples} computes some examples.

\section{Preliminaries}\label{Pre}
\subsection{Lie superalgebras}
Throughout this paper, the ground field is $\mathbb{C}$. The following preliminaries are taken from \cite{K77, K78}. A {\it superspace} over $\mathbb{C}$ is a $\mathbb{Z}_2$-graded vector space $V = V_{\bar{0}} \oplus V_{\bar{1}}$. Let $p(a)$ be the degree of a homogeneous element $a$ and call $a$ {\it even} or {\it odd} if $p(a)$ is $0$ or $1$ respectively. A {\it Lie superalgebra} is a superspace $\mathfrak{g} = \mathfrak{g}_{\bar{0}} \oplus \mathfrak{g}_{\bar{1}}$, with a bilinear form $[\cdot, \cdot]$ satisfying the following axioms for all homogeneous elements $a, b, c \in \mathfrak{g}$
\begin{enumerate}
	\item $[a, b] = -(-1)^{p(a)p(b)}[b, a]$;
	\item $[a, [b, c]] = [[a, b], c] + (-1)^{p(a)p(b)}[b, [a, c]]$.
\end{enumerate}
All simple finite-dimensional Lie superalgebras have been classified by V. Kac in \cite{K75}. A {\it basic classical} Lie superalgebra is one of
\begin{enumerate}
	\item the simple Lie algebras, or
	\item the Lie superalgebras $A(m, n), B(m, n), C(n), D(m, n), D(2,1; \alpha), F(4),$ and $G(3)$.
\end{enumerate}
The reader is referred to \cite{K75, K77, M12} for the constructions of these Lie superalgebras.

Let $\mathfrak{g}$ be a basic classical Lie superalgebra and let $H$ be the Cartan subalgebra of $\mathfrak{g}$ consisting of diagonal matrices. For $\alpha \in H^*\setminus \{0\}$, let $\mathfrak{g}_\alpha = \{x \in \mathfrak{g} \mid [h, x] = \alpha(h)x, \forall h\in H\}$. We call $\alpha$ a {\it root} if $\mathfrak{g}_\alpha \neq 0$. A root $\alpha$ is called {\it even} (respectively {\it odd}) if $\mathfrak{g}_\alpha \cap \mathfrak{g}_{\bar{0}} \neq 0$ (respectively $\mathfrak{g}_\alpha \cap \mathfrak{g}_{\bar{1}} \neq 0$). The set of all roots is $\Delta = \Delta_0 \cup \Delta_1$ where $\Delta_0$ is the set of all even roots and $\Delta_1$ is the set of all odd roots. We also need the set $\overline{\Delta}_1 = \{\alpha \in \Delta_1 \mid 2\alpha \notin \Delta_0 \}$ and we fix an invariant non-degenerate bilinear form $(\cdot, \cdot)$ on $H^*$. 

Fix a Borel subalgebra $B$ of $\mathfrak{g}$. Since the adjoint representation of $H$ in $\mathfrak{g}$ is diagonalizable, we have the following decomposition of $\mathfrak{g}$.
\[
\mathfrak{g} = N^- \oplus H \oplus N^+,  B = H \oplus N^+,
\]
where $N^-$ and $N^+$ are subalgebras and $[H, N^-] \subset N^-, [H, N^+] \subset N^+$. A root $\alpha$ is called {\it positive } if $\mathfrak{g}_\alpha \cap N^+ \neq 0$. We denote by $\Delta^+, \Delta_0^+, \Delta_1^+, \overline{\Delta}_1^+$ the subsets of positive roots in the sets $\Delta, \Delta_0, \Delta_1, \overline{\Delta}_1$ respectively. A positive root $\alpha$ is called {\it simple} if it is indecomposable into a sum of two positive roots. Let $\{\alpha_1, \alpha_2, \ldots, \alpha_r\}$ be the set of all simple roots. Finally, we let $\rho_0$ (respectively $\rho_1$) be the half-sum of all of the even (respectively odd) positive roots and we set $\rho = \rho_0 - \rho_1$. 

From now on, $\mathfrak{g}$ is one of the basic classical Lie superalgebras 
\[
A(m, n), m \neq n, B(m, n), C(n), D(m, n), D(2,1; \alpha), F(4),G(3).
\]
 For $\Lambda \in H^*$, we denote by $V(\Lambda)$ an irreducible representation of $\mathfrak{g}$ ($\mathfrak{g}$-module) with highest weight $\Lambda$. Let $B$ be a distinguished Borel subalgebra of $\mathfrak{g}$ and let $e_i, f_i, h_i, i = 1, 2, \ldots, r$ be the generators of $\mathfrak{g}$ as a contragredient Lie superalgebra. The \textit{numerical marks} for $\Lambda$ are: 
\[
a_i = \Lambda(h_i)
\]
for all $i = 1, 2, \ldots, r$. In fact, $a_i = \frac{2(\Lambda, \alpha_i)}{(\alpha_i, \alpha_i)}$ if $\alpha_i$ is even. Using these numerical marks, we have the conditions for $V(\Lambda)$ being finite-dimensional stated in Proposition~2.3~of~ \cite{K78}.

\subsection{$\mathbb{N}$-typical weights}\label{Ntypical}

Let $\Lambda$ be the highest weight of a finite-dimensional irreducible representation $V(\Lambda)$ of $\mathfrak{g}$. We call the weight $\Lambda$ and the representation $V(\Lambda)$ {\it typical} if 
\[
(\Lambda + \rho, \alpha) \neq 0
\]
for all $\alpha \in \overline{\Delta}_1^+$.
In this paper, we consider a specific weight $\Lambda$ for which $k\Lambda$ is typical and $V(k\Lambda)$ is finite-dimensional for all $k \in \mathbb{N}$ and we call such a weight $\mathbb{N}$-{\it typical}. Note that 
\begin{enumerate}
    \item if $V(\Lambda)$ is finite dimensional, then so is $V(k \Lambda)$ for all $k \in \mathbb{N}$,
    \item if $\Lambda$ is $\mathbb{N}$-typical, then so is $k\Lambda$ for all $k \in \mathbb{N}$.
\end{enumerate}

The complete list of all finite-dimensional irreducible typical $\mathfrak{g}$-representations $V(\Lambda)$ is given in \cite{K78}. This leads us to the following complete criteria for $\Lambda$ being $\mathbb{N}$-typical directly since the form $(\cdot, \cdot)$ is bilinear. For convenience here, we assume that $V(k\Lambda)$ is finite-dimensional for all $k \in \mathbb{N}$. This section quotes from \cite{K78}, but modifies the results there for the case of $\mathbb{N}$-typical representations.

\begin{enumerate}
	\item $\mathfrak{g} = A(m, n), m \neq n$. A weight $\Lambda$ is $\mathbb{N}$-typical if and only if 
	\[
	a_{m+1} \neq \sum_{t = m + 2}^{j}a_t - \sum_{t = i}^{m} a_t - \frac{ 2m}{k} - \frac{2}{k} + \frac{i}{k} + \frac{j}{k}
	\]
	for all $1 \le i \le m + 1 \le j \le m + n + 1$ and $k \in \mathbb{N}$.
	
    \item $\mathfrak{g} = C(n)$. A weight $\Lambda$ is $\mathbb{N}$-typical if and only if 
	\begin{align*}
	a_1 &\neq \sum_{t = 2}^{i} a_t + \frac{i}{k} - \frac{1}{k} \\
	a_1 &\neq \sum_{t = 2}^{i} a_t + 2 \sum_{t = i + 1}^{n} a_t + \frac{2n}{k} - \frac{i}{k} - \frac{1}{k},
	\end{align*} 
	for all $1 \le i \le n - 1$ and $k \in \mathbb{N}$.
	
\item $\mathfrak{g} = B(m, n), m \neq 0$. A weight $\Lambda$ is $\mathbb{N}$-typical if and only if 
	\begin{align*}
&\sum_{t = i}^{n}a_t - \sum_{t = n + 1}^{j} a_t + \frac{2n}{k} - \frac{i}{k} - \frac{j}{k} \neq 0 \text{ and } \\
&\sum_{t = i}^{n}a_t - \sum_{t = n + 1}^{j} a_t - 2 \sum_{t = j + 1}^{m + n -1} a_t - a_{m+n} - \frac{i}{k} + \frac{j}{k} - \frac{2m}{k} + \frac{1}{k} \neq 0,
\end{align*} 
for all $1 \le i \le n \le j \le m + n - 1$ and $k \in \mathbb{N}$.
\item $\mathfrak{g} = B(0, n)$. Any weight $\Lambda$ is $\mathbb{N}$-typical.
\item $\mathfrak{g} = D(m, n)$. A weight $\Lambda$ is $\mathbb{N}$-typical if and only if all of the following hold
\begin{enumerate}
	\item $ \displaystyle \sum_{t = i}^{n}a_t - \sum_{t = n + 1}^{j} a_t + \frac{2n}{k} - \frac{i}{k} - \frac{j}{k} \neq 0 $ \\
	for all $ 1 \le i \le n \le j \le m + n - 1$ and $k \in \mathbb{N}$.
	\item $\displaystyle \sum_{t = i}^{n} a_t - \sum_{t = n + 1}^{m + n - 2}a_t - a_{m + n} + \frac{n}{k} -\frac{m}{k} -\frac{i}{k} + \frac{1}{k} \neq 0$ \\
		for all $ 1 \le i \le n$ and $k \in \mathbb{N}$.
	\item $\displaystyle \sum_{t = i}^{n}a_t - \sum_{t = n + 1}^{j}a_t - 2\sum_{t = j + 1}^{m + n - 2}a_t - a_{m + n - 1} - a_{m + n} - \frac{i}{k} + \frac{j}{k} - \frac{2m}{k} \neq 0$ \\
	for all $ 1 \le i \le n \le j \le m + n - 2$ and $k \in \mathbb{N}$.
\end{enumerate}
\item $\mathfrak{g} = D(2, 1; \alpha)$. A weight $\Lambda$ is $\mathbb{N}$-typical if and only if
\[
a_1 \neq 0, a_1 \neq a_2 + \frac{1}{k}, a_1 \neq \alpha (a_3 + \frac{1}{k}) \text{ and } a_1 \neq a_2 + \alpha (a_3 + \frac{1}{k}) + \frac{1}{k} 
\]
for all $k \in \mathbb{N}$.
\item $\mathfrak{g} = G(3)$. A weight $\Lambda$ is $\mathbb{N}$-typical if and only if
\begin{align*}
&a_1 \neq 0, a_1 \neq a_2 + \frac{1}{k}, a_1 \neq a_2 + 3a_3 + \frac{4}{k}, a_1 \neq 3a_2 + 3a_3 + \frac{6}{k},\\ & a_1 \neq 3a_2 + 6a_3 + \frac{9}{k}, a_1 \neq 4a_2 + 6a_3 + \frac{10}{k}
\end{align*}
for all $k \in \mathbb{N}$.
\item $\mathfrak{g} = F(4)$. A weight $\Lambda$ is $\mathbb{N}$-typical if and only if	
\begin{align*}
&a_1 \neq 0, a_1 \neq a_2 + \frac{1}{k}, a_1 \neq a_2 + 2a_3 + \frac{3}{k}, a_1 \neq 2a_2 + 2a_3 + \frac{4}{k} \\
&a_1 \neq a_2 + 2a_3 + 2a_4 + \frac{5}{k}, a_1 \neq 2a_2 + 2a_3 + 2a_4 + \frac{6}{k} \\
&a_1 \neq 2a_2 + 4a_3 + 2a_4 + \frac{8}{k}, a_1 \neq 3a_2 + 4a_3 + 2a_4 + \frac{9}{k}
\end{align*}
for all $k \in \mathbb{N}$.
\end{enumerate}

\begin{Remark}
	We use the notation that $\displaystyle \sum_{t = r}^{s} a_t = 0$ if $s < r$.
\end{Remark}

\begin{Example}\label{example:N-typical}
Let $\mathfrak{g} = \mathfrak{sl}(2,1)$ and let $\Lambda$ be a weight such that $a_1 =\Lambda(h_1) = 1$ and $a_2=\Lambda(h_2) = -1$. Then $V(\Lambda)$ is a finite dimensional typical $\mathfrak{g}$-module and $\Lambda$ is $\mathbb{N}$-typical.
\end{Example}

\begin{Example}\label{example:fails-N-typical}
Let $\mathfrak{g} = \mathfrak{sl}(2,1)$ and let $\Lambda$ be a weight such that $a_1 = \Lambda(h_1) = 0$ and $a_2 = \Lambda(h_2) = \frac{-1}{2}$. Then $V(\Lambda)$ is a finite dimensional typical $\mathfrak{g}$-module, but $V(2\Lambda)$ is not typical. Therefore $\Lambda$ is not an $\mathbb{N}$-typical weight.
\end{Example}

The main result pertaining to Hilbert series in the next section involves two  types of well-known polynomials, the elementary symmetric polynomials and the Eulerian polynomials. Here we recall the latter polynomials \cite{S2012}. The $j$th {\it Eulerian polynomial} is
\[
A_n(t) := \sum_{k = 0}^{n - 1} A(n, k) t^k,  \ \ \ A_0(t) = 1
\]
where $A(n, k)$ is the number of permutations of the numbers $1$ to $n$ in which exactly $k$ elements are greater than the previous element. These numbers are called {\it Eulerian numbers} and they can be computed from an explicit formula
\begin{equation*}
    A(n, k) = \sum_{r = 0}^{k}(-1)^r \binom{n+1}{r} (k + 1 - r)^n.
\end{equation*}
Note that these numbers can be arranged into a triangle, called {\it Euler's triangle}. 
\begin{align*}
\begin{matrix}
 \ \  \ \ n = 1 \hskip 1.5 cm & & & & & 1 & & & & \\
 \ \  \ \ n = 2 \hskip 1.5 cm &  & & & 1 &  & 1 & & &  \\
 \ \  \ \ n = 3 \hskip 1.5 cm &  & & 1 &  & 4 &  & 1 & & \\
 \ \  \ \ n = 4 \hskip 1.5 cm  &  & 1 &  & 11 &  & 11  &  & 1 & \\
 \ \  \ \ n = 5 \hskip 1.5 cm   & 1 &  & 26  &  & 66 &  & 26 &  & 1 \\
\vdots \hskip 1 cm  & &  &  &  &  \vdots &  & &
\end{matrix}
\end{align*}

\section{Hilbert Series of $\mathbb{N}$-typical representations}\label{section:main-results}

In this section, we present a closed form for the Hilbert series of $\mathbb{N}$-typical representations for $\mathfrak{g}$ defined by
\[
H_\Lambda(q) = \sum_{n \geq 0} \dim (V(n \Lambda)) q^n,
\]
where $V(n\Lambda)$ is a finite dimensional typical $\mathfrak{g}$-module, i.e. $\Lambda$ is an $\mathbb{N}$-typical weight.

\begin{Theorem}\label{mainthm}
	Let $\Lambda$ be an $\mathbb{N}$-typical weight. Then 
	\[
	H_\Lambda(q) = 2^{d_1}\prod_{i = 1}^{d_0}(1 - c_1(\beta_i))\sum_{j = 0}^{d_0}e_j \Big(\frac{c_\Lambda(\beta_1)}{1 - c_1(\beta_1)}, \ldots, \frac{c_\Lambda(\beta_{d_0})}{1 - c_1(\beta_{d_0})} \Big)A_j(q) \frac{q}{(1 - q)^{j + 1}},
	\] 
		where $d_0 = |\Delta_0^+|, d_1 = |\Delta_1^+|, c_1(\alpha) = \frac{(\rho_1, \alpha)}{(\rho_0, \alpha)}, c_\Lambda(\alpha) = \frac{(\Lambda, \alpha)}{(\rho_0, \alpha)}$, $e_j$ is the $j$th elementary symmetric polynomial in $d_0$ variables, $\Delta_0^+ = \{\beta_1, \ldots, \beta_{d_0}\}$, and $A_j$ is the $j$th Eulerian polynomial.
\end{Theorem}

\begin{proof}
	Note that, by Proposition 2.10 in \cite{K78}, we have 
	\begin{align}
	\dim V(k \Lambda) = &2^{d_1} \prod\limits_{\alpha \in \Delta_0^+}\frac{(k\Lambda + \rho, \alpha)}{(\rho_0, \alpha)} \nonumber \\
	= &2^{d_1} \prod\limits_{\alpha \in \Delta_0^+}\frac{(k\Lambda + \rho_0 - \rho_1, \alpha)}{(\rho_0, \alpha)} \nonumber \\
	= &2^{d_1} \prod\limits_{\alpha \in \Delta_0^+}(1 - c_1(\alpha) + k c_\Lambda(\alpha)) \label{dim} \text{ because } (\cdot, \cdot) \text{ is bilinear}.
	\end{align}
	Thus, 
	\begin{equation}\label{HS}
	H_{\Lambda}(q) = \sum_{k \geq 0}\Big(2^{d_1} \prod_{\alpha \in \Delta_0^+}(1 - c_1(\alpha) + k c_\Lambda(\alpha))\Big) q^k.
	\end{equation}
	Ignoring $2^{d_1}$, the product in (\ref{dim}) is a polynomial in $k$. Then
	\begin{align*}
	\prod\limits_{\alpha \in \Delta_0^+}(1 - c_1(\alpha) + k c_\Lambda(\alpha)) &= \prod\limits_{i = 1}^{d_0}(1 - c_1(\beta_i)) \prod\limits_{j = 1}^{d_0}\Big(1 + \frac{kc_\Lambda(\beta_j)}{1 - c_1(\beta_j)}\Big) \\
	&=   \prod\limits_{i = 1}^{d_0}(1 - c_1(\beta_i)) \sum_{j=0}^{d_0}e_j\Big(\frac{c_\Lambda(\beta_1)}{1 - c_1(\beta_1)}, \ldots, \frac{c_\Lambda(\beta_{d_0})}{1 - c_1(\beta_{d_0})} \Big)k^j,
	\end{align*}
	where $e_j$ is the $j$th elementary symmetric polynomial in $d_0$ variables. The series (\ref{HS}) becomes
	\begin{align}\label{polyn}
	H_{\Lambda}(q) = 2^{d_1}\prod\limits_{i = 1}^{d_0}(1 - c_1(\beta_i)) \sum_{j=0}^{d_0}e_j\Big(\frac{c_\Lambda(\beta_1)}{1 - c_1(\beta_1)}, \ldots, \frac{c_\Lambda(\beta_{d_0})}{1 - c_1(\beta_{d_0})} \Big)\sum_{k \geq 0} k^j q^k.
	\end{align}
	We must understand the sum 
	\[
	\sum_{k \geq 0} k^j q^k,
	\]
	which has a nice story on its own. For example, it is connected to \textit{Eulerian polynomials} (see \cite[page 22]{S2012}). A closed form for this interesting series can be derived as in \cite{GW11}. We recall the details. Define
	\[
	f_j(q) = \sum_{k \geq 0} k^j q^k
	\]
	and note that $f_0(q) = \displaystyle\sum_{k \geq 0}q^k = \frac{1}{1 - q}$.
	Applying the differential operator $q \frac{d}{dq}$ to $f_{j - 1}(q)$ gives us $f_j(q)$. Applying the operators successively to $f_0(q)$, we have $f_j(q) = (q\frac{d}{dq})^j\frac{1}{1 - q}$. The expression of $(q\frac{d}{dq})^j\frac{1}{1 - q}$ is well-known. In fact, 
	\[
	\left(q\frac{d}{dq}\right)^j\frac{1}{1 - q} = A_j(q)\frac{q}{(1 - q)^{j + 1}},
	\]
	where $A_j$ is the $j$th Eulerian polynomial as discussed in Section \ref{Pre}.
	Therefore, we have shown 
	\[
		H_\Lambda(q) = 2^{d_1}\prod_{i = 1}^{d_0}(1 - c_1(\beta_i))\sum_{j = 0}^{d_0}e_j \Big(\frac{c_\Lambda(\beta_1)}{1 - c_1(\beta_1)}, \ldots, \frac{c_\Lambda(\beta_{d_0})}{1 - c_1(\beta_{d_0})} \Big)A_j(q) \frac{q}{(1 - q)^{j + 1}}
	\]
	as desired.
\end{proof}

\begin{Corollary}\label{thm:main-theorem}
    Let $\Lambda$ be an $\mathbb{N}$-typical weight. Then 
    \begin{equation*}
    H_\Lambda (q) = h_{\Lambda}\left(q \frac{d}{dq}\right) \frac{1}{1-q},
\end{equation*}
where $h_{\Lambda}(t)$ is a polynomial in one variable which factors as a product of affine linear forms over $\mathbb{Q}$ as in
\begin{equation*}
    h_{\Lambda}(t) = 2^{|\Delta_1^+|} \prod_{\alpha \in \Delta_0^+} \left( 1 - \frac{(\rho_1,\alpha)}{(\rho_0,\alpha)} + \frac{(\Lambda,\alpha)}{(\rho_0,\alpha)} t \right).
\end{equation*}
\end{Corollary}

\begin{proof}
	It follows from the expression 	
	\begin{align*}
H_{\Lambda}(q)& = 2^{d_1}\prod\limits_{i = 1}^{d_0}(1 - c_1(\beta_i)) \sum_{j=0}^{d_0}e_j\Big(\frac{c_\Lambda(\beta_1)}{1 - c_1(\beta_1)}, \ldots, \frac{c_\Lambda(\beta_{d_0})}{1 - c_1(\beta_{d_0})} \Big)\sum_{k \geq 0} k^j q^k \\
& = 2^{d_1}\prod\limits_{i = 1}^{d_0}(1 - c_1(\beta_i)) \sum_{j=0}^{d_0}e_j\Big(\frac{c_\Lambda(\beta_1)}{1 - c_1(\beta_1)}, \ldots, \frac{c_\Lambda(\beta_{d_0})}{1 - c_1(\beta_{d_0})} \Big)\left(q\frac{d}{dq}\right)^j\frac{1}{1 - q} \\
& =  2^{d_1}\prod\limits_{i = 1}^{d_0}(1 - c_1(\beta_i)) \sum_{j=0}^{d_0}e_j\Big(\frac{c_\Lambda(\beta_1)}{1 - c_1(\beta_1)}, \ldots, \frac{c_\Lambda(\beta_{d_0})}{1 - c_1(\beta_{d_0})} \Big)\left(q\frac{d}{dq}\right)^j \sum_{k \geq 0} q^k.
\end{align*}
The last expression differs from (\ref{polyn}) by the substitution $k^j \mapsto (q\frac{d}{dq})^j$. This proves the corollary.
\end{proof}

\begin{Remark}
	The dimension $\dim (V(\Lambda))$ can be recovered by taking $ \dfrac{d}{dq}$ to the R.H.S of the series in the above theorem or corollary and setting $q = 0$.
\end{Remark}

\section{Examples}\label{section:examples}

In this section, we compute several examples explicitly. In all cases the dimensions predicted by our Hilbert series match those from the dimension formula given by Kac in \cite[page 619]{K78}, as they should. Recall that by choosing a basis so that the Cartan $H$ corresponds to diagonal matrices, the root systems of $\mathfrak{g}$ can be described using the linear functionals which extract the entries along the diagonal. In this section we denote these linear functionals by $e_i$ and $d_j$ with $(e_i,e_j) = \delta_{ij}$ and $(d_i,d_j) = -\delta_{ij}$ and $(e_i,d_k) = 0$. Previously $e_i$ meant an elementary symmetric polynomial, but the notation should be clear from context. Then we can write a highest weight $\Lambda = \sum \lambda_i e_i + \sum \mu_j d_j$ so that the $\lambda_i$ and $\mu_j$ are coefficients of $\Lambda$ in the $e_i,d_j$ basis of $H^*$.

\begin{Example}
We consider an example of the Hilbert series $H_{\Lambda}(q)$ where $\Lambda$ is an $\mathbb{N}$-typical weight of $\mathfrak{g} = \mathfrak{sl}(2, 1)$. With the above notation, the simple root system is
\[
\{ \alpha_1 = e_1 - e_2, \hspace{0.5cm} \alpha_2 = e_2 - d_1  \}.
\]
Moreover, $\Delta_0^+ = \{\alpha_1 \}$ and $\Delta_1^+ = \{e_1 - d_1, \alpha_2 \}$.
Let $V(\Lambda)$ be a $\mathfrak{g}$-module with the highest weight $\Lambda$ where $\Lambda$ is an $\mathbb{N}$-typical weight, i.e., $a_1$ is a non-negative integer, $a_2 \neq -a_1 - \dfrac{1}{k}$ for all $k \in \mathbb{N}$ and $a_2 \neq 0$ (cf. Section \ref{Ntypical}). Then 
\[
H_{\Lambda}(q) = 2^2 \prod\limits_{\alpha \in \{\alpha_1 \}} \left(1 - c_1(\alpha) + c_{\Lambda}(\alpha)q \frac{d}{dq} \right) \frac{1}{1 - q}.
\]
Note that $c_1(\alpha_1) = 0$ and $c_\Lambda(\alpha_1) =  a_1$. Therefore, 
\begin{align*}
    H_{\Lambda}(q) = &4 \prod\limits_{\alpha \in \{\alpha_1 \}} \left(1  +  a_1 q \frac{d}{dq} \right) \frac{1}{1 - q} \\
	=& 4\Big( \frac{1 - q + a_1q}{(1 - q)^2}\Big).
\end{align*}

Below we compute several examples of dominant integral weights $\Lambda$ which are also $\mathbb{N}$-typical, and compute their Hilbert series according to Theorem \ref{mainthm} and Corollary \ref{thm:main-theorem}.
\begin{equation*}
\begin{tabular}{lll}
 $\Lambda, (a_1,a_2)$ & Hilbert Series &   \\ \hline
$\begin{tabular}{l}
$e_{1} + e_{2}$ \\
$\left(0, 1\right)$ \\
\end{tabular}$ & $\frac{4}{1 - q}$ & $4 + 4 q + 4 q^{2} + 4 q^{3} + 4 q^{4} + \mathcal{O}\left(q^{5}\right)$ \\
$\begin{tabular}{l}
$2 \, e_{1} + e_{2}$ \\
$\left(1, 1\right)$ \\
\end{tabular}$ & $\frac{4}{{\left(1 - q\right)}^{2}}$ & $4 + 8 q + 12 q^{2} + 16 q^{3} + 20 q^{4} + \mathcal{O}\left(q^{5}\right)$ \\
$\begin{tabular}{l}
$3 \, e_{1} + e_{2}$ \\
$\left(2, 1\right)$ \\
\end{tabular}$ & $\frac{4 \, {\left(q + 1\right)}}{{\left(1-q\right)}^{2}}$ & $4 + 12 q + 20 q^{2} + 28 q^{3} + 36 q^{4} + \mathcal{O}\left(q^{5}\right)$ \\
$\begin{tabular}{l}
$4 \, e_{1} + e_{2}$ \\
$\left(3, 1\right)$ \\
\end{tabular}$ & $\frac{4 \, {\left(2 \, q + 1\right)}}{{\left(1-q\right)}^{2}}$ & $4 + 16 q + 28 q^{2} + 40 q^{3} + 52 q^{4} + \mathcal{O}\left(q^{5}\right)$ \\
$\begin{tabular}{l}
$5 \, e_{1} + e_{2}$ \\
$\left(4, 1\right)$ \\
\end{tabular}$ & $\frac{4 \, {\left(3 \, q + 1\right)}}{{\left(1-q\right)}^{2}}$ & $4 + 20 q + 36 q^{2} + 52 q^{3} + 68 q^{4} + \mathcal{O}\left(q^{5}\right)$ \\
$\begin{tabular}{l}
$6 \, e_{1} + e_{2}$ \\
$\left(5, 1\right)$ \\
\end{tabular}$ & $\frac{4 \, {\left(4 \, q + 1\right)}}{{\left(1-q\right)}^{2}}$ & $4 + 24 q + 44 q^{2} + 64 q^{3} + 84 q^{4} + \mathcal{O}\left(q^{5}\right)$ \\
\end{tabular}
\end{equation*}
\end{Example}

\begin{Example}
Consider $\mathfrak{g} = \mathfrak{sl}(3,2)$. We display below all the $\mathbb{N}$-typical weights with numerical marks $(a_1,a_2,a_3,a_4)$ for $a_i \in \{0,1,2\}$ but fixing the single odd mark $a_3 = 1$.

\begin{footnotesize}
\begin{equation*}
\begin{tabular}{lll}
$\Lambda, \left(a_1,a_2,1,a_4\right)$ & Hilbert Series &   \\ \hline
$\begin{tabular}{l}
\footnotesize$-d_{2} + 2 \, e_{1} + 2 \, e_{2} + e_{3}$ \\
$\left(0, 1, 1, 1\right)$ \\
\end{tabular}$ & $\frac{64 \, {\left(2 \, q + 1\right)}}{{- \left(1-q \right)}^{4}}$ & $64 + 384 q + 1152 q^{2} + 2560 q^{3} + \cdots$ \\
$\begin{tabular}{l}
\footnotesize$-d_{2} + 3 \, e_{1} + 3 \, e_{2} + e_{3}$ \\
$\left(0, 2, 1, 1\right)$ \\
\end{tabular}$ & $\frac{64 \, {\left(3 \, q^{2} + 8 \, q + 1\right)}}{{ \left(1-q \right)}^{4}}$ & $64 + 768 q + 2880 q^{2} + 7168 q^{3} + \cdots$ \\
$\begin{tabular}{l}
\footnotesize$-2 \, d_{2} + 3 \, e_{1} + 3 \, e_{2} + e_{3}$ \\
$\left(0, 2, 1, 2\right)$ \\
\end{tabular}$ & $\frac{64 \, {\left(9 \, q^{2} + 14 \, q + 1\right)}}{{ \left(1-q \right)}^{4}}$ & $64 + 1152 q + 4800 q^{2} + 12544 q^{3} + \cdots$ \\
$\begin{tabular}{l}
\footnotesize$-2 \, d_{2} + 2 \, e_{1} + e_{2} + e_{3}$ \\
$\left(1, 0, 1, 2\right)$ \\
\end{tabular}$ & $\frac{64 \, {\left(5 \, q + 1\right)}}{{ \left(1-q \right)}^{4}}$ & $64 + 576 q + 1920 q^{2} + 4480 q^{3} + \cdots$ \\
$\begin{tabular}{l}
\footnotesize$-d_{2} + 3 \, e_{1} + 2 \, e_{2} + e_{3}$ \\
$\left(1, 1, 1, 1\right)$ \\
\end{tabular}$ & $\frac{64 \, {\left(q^{2} + 10 \, q + 1\right)} {\left(q + 1\right)}}{{ \left(1-q \right)}^{5}}$ & $64 + 1024 q + 5184 q^{2} + 16384 q^{3} + \cdots$ \\
$\begin{tabular}{l}
\footnotesize$-d_{2} + 4 \, e_{1} + 3 \, e_{2} + e_{3}$ \\
$\left(1, 2, 1, 1\right)$ \\
\end{tabular}$ & $\frac{64 \, {\left(6 \, q^{3} + 40 \, q^{2} + 25 \, q + 1\right)}}{{ \left(1-q \right)}^{5}}$ & $64 + 1920 q + 11520 q^{2} + 39424 q^{3} + \cdots$ \\
$\begin{tabular}{l}
\footnotesize$-2 \, d_{2} + 4 \, e_{1} + 3 \, e_{2} + e_{3}$ \\
$\left(1, 2, 1, 2\right)$ \\
\end{tabular}$ & $\frac{64 \, {\left(9 \, q^{2} + 38 \, q + 1\right)} {\left(2 \, q + 1\right)}}{{\left(1-q \right)}^{5}}$ & $64 + 2880 q + 19200 q^{2} + 68992 q^{3} + \cdots$ \\
$\begin{tabular}{l}
\footnotesize$-2 \, d_{2} + 3 \, e_{1} + e_{2} + e_{3}$ \\
$\left(2, 0, 1, 2\right)$ \\
\end{tabular}$ & $\frac{64 \, {\left(9 \, q^{2} + 14 \, q + 1\right)}}{{ \left(1-q \right)}^{4}}$ & $64 + 1152 q + 4800 q^{2} + 12544 q^{3} + \cdots$ \\
$\begin{tabular}{l}
\footnotesize$-d_{2} + 4 \, e_{1} + 2 \, e_{2} + e_{3}$ \\
$\left(2, 1, 1, 1\right)$ \\
\end{tabular}$ & $\frac{64 \, {\left(6 \, q^{3} + 40 \, q^{2} + 25 \, q + 1\right)}}{{ \left(1-q \right)}^{5}}$ & $64 + 1920 q + 11520 q^{2} + 39424 q^{3} + \cdots$ \\
$\begin{tabular}{l}
\footnotesize$-d_{2} + 5 \, e_{1} + 3 \, e_{2} + e_{3}$ \\
$\left(2, 2, 1, 1\right)$ \\
\end{tabular}$ & $\frac{64 \, {\left(27 \, q^{3} + 115 \, q^{2} + 49 \, q + 1\right)}}{{ \left(1-q \right)}^{5}}$ & $64 + 3456 q + 24000 q^{2} + 87808 q^{3} + \cdots$ \\
$\begin{tabular}{l}
\footnotesize$-2 \, d_{2} + 5 \, e_{1} + 3 \, e_{2} + e_{3}$ \\
$\left(2, 2, 1, 2\right)$ \\
\end{tabular}$ & $\frac{64 \, {\left(q^{4} + 76 \, q^{3} + 230 \, q^{2} + 76 \, q + 1\right)}}{{ \left(1-q \right)}^{5}}$ & $64 + 5184 q + 40000 q^{2} + 153664 q^{3} + \cdots$ \\
\end{tabular}
\end{equation*}
\end{footnotesize}
\end{Example}

\begin{Example}
We consider an example of the Hilbert series $H_{\Lambda}(q)$ where $\Lambda$ is an $\mathbb{N}$-typical weight of $\mathfrak{g} = \mathfrak{sl}(4, 1)$. In the notation above, the simple root system is given by
\begin{equation*}
    \{ \alpha_1 = e_1 - e_2, \hspace{0.3cm} \alpha_2 = e_2 - e_3, \hspace{0.3cm} \alpha_3 = e_3 - e_4, \hspace{0.3cm} \alpha_4 = e_4 - d_1\}.
\end{equation*}
Moreover, $\Delta_0^+ = \{\alpha_1, \alpha_2, \alpha_3, \alpha_1 + \alpha_2, \alpha_2 + \alpha_3, \alpha_1 + \alpha_2 + \alpha_3 \}$ and $\Delta_1^+ = \{e_i - d_1 \}, i = 1, 2, 3, 4$.
Note that $\rho_0 = \frac{1}{2}(3e_1 + e_2 - e_3 - 3 e_4)$ and $\rho_1 = \frac{1}{2}(e_1 + e_2 + e_3 + e_4 - 4 d_1)$. Thus, $c_1(\alpha) = \frac{(\rho_1, \alpha)}{\rho_0, \alpha} = 0, \forall \alpha \in \Delta_0^+$ and $(\rho_0, \alpha_i) = 1, i = 1, 2, 3$. Let $\Lambda = e_1 + e_2 + d_1$. Then $V(\Lambda)$ is a finite dimensional $\mathfrak{g}$-module with the highest weight $\Lambda$ where $\Lambda$ is an $\mathbb{N}$-typical weight and
\begin{equation*}
    (\Lambda, \alpha_1) = 0, (\Lambda, \alpha_2) = 1, (\Lambda, \alpha_3) = 0.
\end{equation*}
Therefore
\begin{align*}
    H_{\Lambda}(q) &= 2^4 \prod\limits_{\alpha \in \Delta_0^+ } \left( 1 + c_{\Lambda}(\alpha)q \frac{d}{dq} \right) \frac{1}{1 - q} \\
    &= 16 \left( 1 + q \frac{d}{dq} \right) \left( 1 + \frac{1}{2}q \frac{d}{dq} \right) \left( 1 + \frac{1}{2}q \frac{d}{dq} \right) \left(1 + \frac{1}{3}q \frac{d}{dq} \right) \frac{1}{1 - q} \\
    &= \frac{16(1 + q)}{(1 - q)^5}.
\end{align*}
\end{Example}

\section{A motivating geometric example}\label{section:grassmannian}

For a semisimple algebraic group $G$ over $\mathbb{C}$, the projective varieties with a transitive $G$ action correspond to quotients $G/P$ for a parabolic subgroup $P$. The equivariant projective embeddings of these varieties are in correspondence with dominant integral weights $\lambda$. The homogeneous coordinate ring of such a projective embedding has an associated Hilbert series $H_\lambda(q)$ recording the dimensions of its graded components, all but finitely many of whose coefficients are given by evaluation of the Hilbert polynomial $h_\lambda(t)$, from which the degree and dimension of the variety may also be read. This equivariant embedding arises as the unique closed orbit of $G$ on a projective space associated to the irreducible representation $V_\lambda$. This representation itself may be realized as the space of global sections on the variety $G/P$. For precise statements of these results see \cite{FH1991}.

Much of this story carries over to the case of Lie supergroups and algebraic supergeometry, whose theory has been greatly developed recently \cite{CCF2011, DeligneMorgan, Manin, Masuoka2009, Masuoka2012}. Indeed, every irreducible finite dimensional representation of a simple simply connected complex supergroup $G$ is realized analogously inside the space of global sections of certain line bundles \cite{Penkov} (super Borel-Weil-Bott theory). If the superalgebra $\mathcal{O}(G/P)$ of global sections of this line bundle is \textit{very ample}, i.e. it is generated in degree one, then $G/P$ admits a projective embedding and $\mathcal{O}(G/P)$ is called the coordinate superalgebra of this embedding. In this case, the dimension of $V(n \Lambda)$ equals the dimension of the $n$th graded component $\mathcal{O}(G/P)_n$. For precise statements see~\cite[Proposition~3.6]{CFV2018} and also \cite{Fioresi2015}. For an introduction to supergeometry see~\cite{Varadarajan}.

In this section we briefly consider atypical representations, motivated by the example of the Pl\"ucker embedding of the Grassmannian of $2|0$ subspaces in $\mathbb{C}^{4|1}$, denoted $Gr_{2|0}(4|1)$. This Grassmannian corresponds to $G/P$ with $G = SL(4,1)$ and $P$ the block upper triangular subgroup given by $\begin{bmatrix} A & B\\0 & D \end{bmatrix}$ where $A$ is $2 \times 2$, $B$ is $2 \times 3$, and $D$ is $3 \times 3$. For more details see \cite{CFL} where the big cell in this Grassmannian is called \textit{chiral super Minkowski space} and elements of its coordinate superalgebra are called \textit{chiral superfields}.

The Grassmannian supervariety does not admit a projective embedding in general \cite{Manin}, but in the case of $Gr_{2|0}(4|1)$ it does. In \cite{CFL}, the authors determine the coordinate superalgebra of $Gr_{2|0}(4|1)$ with respect to what they call its Pl\"ucker embedding in $\mathbb{P}\left( \bigwedge^2 \mathbb{C}^{4|1} \right)$, a construction they also give. We refer the reader to \cite{CFL} for all the required details and definitions regarding this example, but note that the relevant Lie superalgebra and highest weight are $\mathfrak{sl}(4,1)$ and $\Lambda = e_1 + e_2$. Our purpose here is to mention this example as motivation for future study.

The highest weight $\Lambda = e_1 + e_2$ is dominant integral, having numerical marks $(a_1,a_2,a_3,a_4) = (0,1,0,0)$. However, $\Lambda$ is not $\mathbb{N}$-typical since the condition $(\Lambda + \rho, \alpha) \neq 0$ is violated for the root $\beta_{41} := e_4 - d_1 \in \Bar{ \Delta}_1^+$. Since this condition is only violated for one root, $\Lambda$ is called \textit{singly atypical}. \color{black} In fact, we found that $k\Lambda$ is singly atypical for $k \in \mathbb{N}$. \color{black}

In \cite[Theorem 5]{Jeugt1995} a formula for the dimensions of representations corresponding to singly atypical dominant integral weights of $\mathfrak{sl}(m,n)$ is given. If $\Lambda = \sum \lambda_i e_i + \sum \mu_j d_j$ is a dominant integral weight where the root $\beta_{k\ell} = e_k - d_\ell$ is the only root in $\Bar{\Delta}_1^+$ with $(\Lambda + \rho, \beta_{k\ell}) = 0$ then $\Lambda$ is singly atypical and the formula giving its dimension is
\begin{multline}\label{eqn:dim-singly-atypical}
    \text{dim} V(\Lambda) = 2^{mn-1} \left( \prod_{\begin{array}{c}
        i<j \\
        i,j \neq k
    \end{array} } \frac{\lambda_i - \lambda_j + j-i}{j-i} \right) \left( \prod_{\begin{array}{c}
        i<j \\
        i,j \neq \ell
    \end{array} } \frac{\mu_i - \mu_j + j-i}{j-i} \right) \cdots \\
    \cdot \frac{(-1)^{n-k-\ell-1} }{(m-k)! (k-1)! (n-\ell)! (\ell-1)!} \cdot \left( \sum_{r=0}^{m+n-2} C_{m+n-2-r} \cdot e_r\left(x_1,\dots,x_{m-1},y_1,\dots,y_{n-1}\right) \right).
\end{multline}
In the formula (\ref{eqn:dim-singly-atypical}), $e_r$ again denotes the $r$th elementary symmetric polynomial in $m+n-2$ variables, and it is evaluated at $x_1,\dots,x_{m-1},y_1,\dots,y_{n-1}$ where $x_i = \lambda_k - \lambda_i + i - k$ and $y_j = \mu_j - \mu_\ell + \ell - j$. Finally, the coefficient $C_{m+n-2-r}$ of $e_r$ is given by the generating function $\frac{2 e^t}{1+e^t}$ where $C_i$ is the coefficient of $t^i$ in its power series expansion. These coefficients $C_i$ are related to the Bernoulli numbers $B_i$ (see OEIS A027641 for numerators and A027642 for denominators) by the formula
\begin{equation*}
    C_i = \frac{2(2^{i+1}-1)}{i+1}B_{i+1}.
\end{equation*}
We made these calculations for $\mathfrak{sl}(4,1)$ with highest weight $\Lambda = e_1 + e_2$ and found that
\begin{equation*}
    \begin{array}{rl}
        \text{dim}V(\Lambda) & =11 \\
        \text{dim}V(2\Lambda) & =46 \\
        \text{dim}V(3\Lambda) & =130 \\
        \text{dim}V(4\Lambda) & =295 \\
        \text{dim}V(5\Lambda) & =581 .
    \end{array}
\end{equation*}
We can compare these dimensions to the coefficients in the Hilbert series given by Theorem \ref{thm:main-theorem}. Since $\Lambda$ fails to be $\mathbb{N}$-typical, the coefficients from our formula provide upper bounds. Indeed, we obtain
\begin{align}
    \frac{16 + 16q}{(1-q)^5} & =16 + 96 q + 320 q^{2} + 800 q^{3} + 1680 q^{4} + 3136 q^{5} + \cdots \nonumber \\
    \sum_{k \in \mathbb{N}} \text{dim} V(k\Lambda)q^k & =1 + 11q + 46q^2 + 130q^3 + 295q^4 + 581q^5 + \cdots. \label{eqn:hilbert-series-grassmannian}
\end{align}
Thus, the actual dimensions are less than or equal to the dimensions predicted by our Hilbert series, which is expected for $\Lambda$ which fails to be $\mathbb{N}$-typical.
\color{black} One could attempt in future work to use formula (\ref{eqn:dim-singly-atypical}) in place of
\begin{equation*}
    2^{|\Delta_1^+|} \prod\limits_{\alpha \in \Delta_0^+}\frac{(k\Lambda + \rho, \alpha)}{(\rho_0, \alpha)}
\end{equation*}
and obtain a closed formula for the Hilbert series of singly atypical weights $\Lambda$ whose every nonnegative integer multiple is also singly atypical. As a consequence of the super Borel-Weil-Bott theorem, such a formula would also record the dimensions of the graded components of the coordinate superalgebra of the corresponding projective embedding, should it exist. Applied to the Pl\"ucker embedding of the super Grassmannian $Gr_{2|0}(4|1)$ such a formula would record the correct dimensions given by (\ref{eqn:hilbert-series-grassmannian}). Thus, there is geometric motivation to develop these ideas further.
\color{black}

We carry out a few calculations to verify that $\text{dim} \, V(n\Lambda) = \text{dim} \, \mathcal{O}(G/P)_n$ in this example, which holds as a consequence of the super Borel-Weil-Bott theorem.
Consider the free algebra
\begin{equation*}
    \mathcal{A} = \mathbb{C}\langle q_{12}, q_{13}, q_{14}, q_{23}, q_{24}, q_{34}, a_{55}, \lambda_1, \lambda_2, \lambda_3, \lambda_4 \rangle.
\end{equation*}
Adjoin relations making the even variables $q_{12}, q_{13}, q_{14}, q_{23}, q_{24}, q_{34}, a_{55}$ commute, the odd variables into Grassmann variables using $\langle \lambda_i \lambda_j + \lambda_j \lambda_i \mid i,j \in \{1,2,3,4\} \rangle$, and lastly making the even and odd variables commute. Next, adjoin the \textit{super Pl\"ucker relations} \cite{CFL, FLbook} generated~by
\begin{equation*}
    \begin{array}{rl}
        q_{12}q_{34} - q_{13}q_{24}+q_{14}q_{23} = 0, &  \\
        q_{ij}\lambda_k - q_{ik}\lambda_j + q_{jk}\lambda_i = 0, & 1\leq i < j < k \leq 4\\
        \lambda_i \lambda_j - a_{55} q_{ij} = 0, & 1\leq i < j \leq 4 \\
        \lambda_i a_{55} = 0, & 1 \leq i \leq 4.
    \end{array}
\end{equation*}
The $11$ degree one generators match the dimension $\text{dim} \, V(\Lambda) = 11$.
There is one additional relation, namely $a_{55} a_{55} = 0$. A careful reading of \cite[p. 213]{FLbook} shows that the Pl\"ucker relations are derived from the decomposibility of $Q=(r+\xi \epsilon_5) \wedge (s+\eta \epsilon_5)$, with $a_{55}$ defined as $\xi\eta$. But $\xi$ and $\eta$ are odd \cite[p. 212]{FLbook} hence $a_{55} a_{55} = \xi\eta\xi\eta=-\xi^2\eta^2=0$.
Using noncommutative Gr\"obner bases techniques \cite{letterplace} we calculated normal forms for all monomials of degree two in the generators.

\begin{equation*}
    \begin{tabular}{llllll}
$q_{12} q_{12}$ & $q_{13} q_{13}$ & $q_{14} q_{14}$ & $q_{23} q_{23}$ & $q_{24} q_{24}$ & $q_{34} q_{34}$ \\
$q_{13} q_{12}$ & $q_{14} q_{13}$ & $q_{24} q_{14}$ & $q_{24} q_{23}$ & $q_{34} q_{24}$ & $-\lambda_{4} \lambda_{3}$ \\
$q_{14} q_{12}$ & $q_{23} q_{13}$ & $q_{34} q_{14}$ & $q_{34} q_{23}$ & $-\lambda_{4} \lambda_{2}$ & $\lambda_{3} q_{34}$ \\
$q_{23} q_{12}$ & $q_{24} q_{13}$ & $-\lambda_{4} \lambda_{1}$ & $-\lambda_{3} \lambda_{2}$ & $\lambda_{2} q_{24}$ & $\lambda_{4} q_{34}$ \\
$q_{24} q_{12}$ & $q_{34} q_{13}$ & $\lambda_{1} q_{14}$ & $\lambda_{2} q_{23}$ & $\lambda_{3} q_{24}$ &  \\
$q_{34} q_{12}$ & $-\lambda_{3} \lambda_{1}$ & $\lambda_{2} q_{14}$ & $\lambda_{3} q_{23}$ & $\lambda_{4} q_{24}$ &  \\
$-\lambda_{2} \lambda_{1}$ & $\lambda_{1} q_{13}$ & $\lambda_{3} q_{14}$ & $\lambda_{4} q_{23}$ &  &  \\
$\lambda_{1} q_{12}$ & $\lambda_{2} q_{13}$ & $\lambda_{4} q_{14}$ &  &  &  \\
$\lambda_{2} q_{12}$ & $\lambda_{3} q_{13}$ &  &  &  &  \\
$\lambda_{3} q_{12}$ & $\lambda_{4} q_{13}$ &  &  &  &  \\
$\lambda_{4} q_{12}$ &  &  &  &  &  \\
\end{tabular}
\end{equation*}
Since $\text{dim} \, V(2\Lambda) = 46$, there should be $46$ linearly independent degree two elements, which matches what we see above. Thus we obtain $\text{dim} \, V(2\Lambda) = 46 = \text{dim} \, \mathcal{O}(G/P)_2$, confirming this consequence of the super Borel-Weil-Bott theorem in this example.
We also checked that there are $130$ linearly independent degree three elements when we include the relation $a_{55} a_{55} = 0$. This matches $\text{dim} \, V(3\Lambda) = 130$, as it should.

In this section, we simply reported the results of a calculation. In conclusion, this example gives geometric motivation to extend our formulas for the Hilbert series to the case of singly atypical representations.

\smallskip
\section*{Acknowledgements}
The second author would like to acknowledge the financial support provided by the Center of Excellence in Theoretical and Computational Science (TaCS-CoE), Faculty of Science, King Mongkut's University of Technology Thonburi (KMUTT). Both authors thank Rita Fioresi and Maria Lled\'o for very helpful discussions.
\smallskip



\begin{thebibliography}{99}


\bibitem{CCF2011}
C. Carmeli, L. Caston, R. Fioresi: \emph{Mathematical Foundation of Super-symmetry, with an appendix with I. Dimitrov}, EMS Ser. Lect. Math., European Math. Soc., Zurich (2011).

\bibitem{CFV2018}
C. Carmeli, R. Fioresi, V. S. Varadarajan: \emph{Super Bundles}, Universe {\bf 4} no. 3 46 (2018).

\bibitem{CFL}
D. Cervantes, R. Fioresi, M.A. Lled\'o: \emph{On Chiral Quantum Superspaces}, in: Ferrara S., Fioresi R., Varadarajan V. (eds) Supersymmetry in Mathematics and Physics. Lecture Notes in Mathematics, vol 2027. Springer, Berlin, Heidelberg (2011).

\bibitem{FLbook}
R. Fioresi and M.A. Lled\'o: \emph{The Minkowski and Conformal Superspaces}, World Scientific, Singapore (2015).

\bibitem{DeligneMorgan}
P. Deligne, J. Morgan: \emph{Notes on supersymmetry} (following J. Bernstein), in: “Quantum fields and strings. A course for mathematicians”, Vol. 1, AMS, (1999).

\bibitem{Fioresi2015}
R. Fioresi: \emph{Quantum homogeneous superspaces and quantum duality principle}. Banach Center Publications {\bf 106} (2015) 59--72.

\bibitem{FH1991}
W. Fulton, J. Harris: \emph{Representation theory, a first course}. Graduate Texts in Mathematics, Springer-Verlag New York, New York vol. 129 (1991).

\bibitem{GW11}
B. H. Gross, N. R. Wallach: \emph{On the Hilbert polynomials and Hilbert series of homogeneous projective varieties}, in: Arithmetic Geometry and Automorphic Forms, in: Adv. Lect. Math. (ALM) {\bf 19} (2011) 253--263.

\bibitem{K75}
V. G. Kac: \emph{Classification of simple Lie superalgebras}, Funktsional. Anal. i Prilozhen. {\bf 9} (1975) 91--92.

\bibitem{K77}
V. G. Kac:
\emph{Lie superalgebras}, Adv. Math. {\bf 26} (1977) 8--96.	

\bibitem{K78}
V. G. Kac: \emph{Representations of classical Lie superalgebras}, Lecture Notes in Math. {\bf 676} (1978) 597--626.

\bibitem{letterplace}
R. La Scala, V. Levandovskyy: \emph{Letterplace ideals and non-commutative Gr\"obner bases}, J. Symb. Comput. {\bf 44} (2009) 1374--1393.

\bibitem{Manin}
Y. I. Manin: \emph{Gauge field theory and complex geometry}; translated by N. Koblitz and J. R. King, Springer-Verlag, Berlin-New York (1988).

\bibitem{Masuoka2009}
A. Masuoka, A. N. Zubkov: \emph{Quotient sheaves of algebraic supergroups are superschemes}, J. Alg. {\bf 348} (2009) 135--170.

\bibitem{Masuoka2012}
A. Masuoka: \emph{Harish-Chandra pairs for algebraic affine supergroup schemes over an arbitrary field}, Transf. Groups {\bf 17} (2012) 1085-–1121.

\bibitem{M12}
I. M. Musson: {\emph Lie Superalgebras and Enveloping Algebras}, Grad. Stud. Math., vol. 131, American Mathematical Society, Providence, RI, (2012).

\bibitem{Penkov}
I. B. Penkov: \emph{Borel-Weil-Bott theory for classical Lie supergroups}, J. Sov. Math. {\bf 51} (1990) 2108--2140.

\bibitem{S2012}
R. P. Stanley: \emph{Enumerative combinatorics}. Volume 1. 
Second edition. Cambridge Studies in Advanced Mathematics, 49. Cambridge University Press, Cambridge (2012).

\bibitem{Jeugt1995}
J. Van der Jeugt: \emph{Dimension formulas for the Lie superalgebra sl(m/n)}. J. Math. Phys {\bf 36} (1995) 605--611.

\bibitem{Varadarajan}
V. S. Varadarajan: \emph{Supersymmetry for mathematicians}: an introduction,
Courant Lecture Notes 1, AMS, (2004).


\end{thebibliography}
\end{document}